\documentclass[12pt]{article}
\usepackage{amsmath,amsthm,amssymb,graphicx,caption,subcaption}
\usepackage[margin=1.1in]{geometry}
\usepackage{tikz,enumitem}
\usepackage{hyperref}
\usepackage{esvect}
\usepackage{float}
\usepackage{amsmath}
\newtheorem{theorem}{Theorem}

\newtheorem{corollary}{Corollary}
\newtheorem{lemma}{Lemma}

\begin{document}

\title{Nontrivial Solutions to a Cubic Identity and the Factorization of $n^2+n+1$}
\author{Hajrudin Fejzi\'c\thanks{Department of Mathematics, California State University San Bernardino. Email: hfejzic@csusb.edu.}}

\maketitle
\renewcommand{\thefootnote}{}
\renewcommand{\thefootnote}{\arabic{footnote}}

\begin{abstract}
We investigate a variation of Nicomachus's identity in which one term in the cubic sum is replaced by a different cube. Specifically, we study the Diophantine identity
\[
\sum_{j=1}^{n} j^3 + x^3 - k^3 = \left( \sum_{j=1}^{n} j + x - k \right)^2
\]
and classify all integer solutions \((k,x,n)\). A full parametric family of nontrivial solutions was introduced in a 2005 paper, along with a conjectural condition for when such solutions exist. We provide a complete proof of this characterization and show it is equivalent to a structural condition on the prime factorization of \( n^2 + n + 1 \).

Our argument connects this identity to classical results in the theory of binary quadratic forms. In particular, we analyze the equation \(a^2 + ab + b^2 = n^2 + n + 1\), interpreting it as a norm in the ring of Eisenstein integers \(\mathbb{Z}[\omega]\), where \(\omega = \frac{1 + \sqrt{-3}}{2}\). This yields a surprising connection between a modified combinatorial identity and the arithmetic of algebraic number fields.
\end{abstract}

\medskip

\section{Introduction}

In a 2005 paper \cite{fejzic2005}, the authors explored a variation of the classical identity
\[
\sigma^3(n) := \sum_{j=1}^{n} j^3 = \left( \sum_{j=1}^{n} j \right)^2 := \big( \sigma(n) \big)^2,
\]
often referred to as Nicomachus's identity. The variant considered what happens when one of the terms in the cubic sum is replaced with a different cube. Specifically, the authors studied integer solutions to the identity
\[
\sigma^3(n) + x^3 - k^3 = \left( \sigma(n) + x - k \right)^2.
\]

They gave a full parametric classification of all integer solutions \( (k,x,n) \). Trivial solutions such as \( (x,x,n) \) and \( (n-1,2,n) \) exist for all \( n \), but the nature of nontrivial solutions turned out to be far richer.

In that same paper, the authors also announced a deeper result: that a nontrivial solution exists if and only if \( N = n^2 + n + 1 \) is not equal to a prime or three times a prime. While the main parametric result was published, the accompanying characterization was stated without proof,
due to its  technical nature, which made it unsuitable for an undergraduate journal.

In this paper, we prove a stronger result: that nontrivial solutions exist if and only if \( N = n^2 + n + 1 \) has at least two prime divisors congruent to \( 1 \mod 3 \), counted with multiplicity. \footnote{That is, if $N$ is divisible by a prime $p$ raised to a power $r$, then $p$ is said to divide $N$ with multiplicity $r$.} We then deduce the original characterization as a corollary. 

A key step in our argument involves analyzing when the quadratic form \( a^2 + ab + b^2 \) represents the quantity \( N = n^2 + n + 1 \), and how many such representations exist. To do so, we rely on classical results from the arithmetic of the Eisenstein integers \( \mathbb{Z}[\omega] \), where \( \omega = \frac{1 + \sqrt{-3}}{2} \). In particular, we use the structure of this ring to count representations, understand orbits under the unit group, and prove that each orbit contains exactly one positive solution. These tools are developed in Section~\ref{sec:Zomega}, and are crucial to the proof of Lemma~\ref{l1} and Theorem~\ref{th2}.

\section{The Modified Cubic Identity and Quadratic Norms}\label{MCI}

We impose natural constraints on $x$ and $k$ in the identity \eqref{e1}, requiring that $x \ne k$, $x \ne 2$, and $k \ne n - 1$, so that the identity does not reduce to one of the trivial cases. Suppose $(k,x,n)$ is a solution to
\begin{equation}\label{e1}
\sigma^3(n) + x^3 - k^3 = \left( \sigma(n) + x - k \right)^2.
\end{equation}

By substituting and simplifying, one obtains the identity:

$$
 x^2 + kx + k^2 = n(n+1) + (x - k).
$$

Introducing new variables $k = a - 1$, $x = b + 1$, this equation becomes:

$$
 a^2 + ab + b^2 = n^2 + n + 1 =: N,
$$

subject to the constraints $a \ne b + 2$, $b \ne 1$, and $a \ne n$, which correspond to the conditions on $x$ and $k$ above.

In addition, we are interested only in \textbf{positive} solutions, meaning those for which both $a > 0$ and $b > 0$.

The connection between the quadratic form \( a^2 + ab + b^2 \) and the norm function in \( \mathbb{Z}[\omega] \), where \( \omega = \frac{1 + \sqrt{-3}}{2} \), is classical and well known to those familiar with algebraic number theory. The following section reviews basic properties of this ring and its norms, primarily for readers who may not be as familiar with this background. Readers comfortable with the arithmetic of \( \mathbb{Z}[\omega] \) may wish to skip ahead to Section~\ref{s1}, where the main argument resumes.

\section{Arithmetic in $\mathbb{Z}[\omega]$}\label{sec:Zomega}

The following theorem is a classical result from the theory of binary quadratic forms, which plays a key role in counting the number of nontrivial solutions to our generalized cubic identity \eqref{e1}.

\begin{theorem}[Representation by Norm]\label{th1}
Let $n = 3^j p_1^{n_1} \cdots p_k^{n_k} q_1^{m_1} \cdots q_d^{m_d}$, where each $p_i \equiv 1 \mod 3$ and each $q_i \equiv 2 \mod 3$ is a rational prime. Then $n$ can be expressed as $a^2 + ab + b^2$ for some integers $a, b$ if and only if all exponents $m_i$ are even. In that case, the number of such representations is

$$
6(n_1+1)(n_2+1)\cdots(n_k+1).
$$

\end{theorem}

This result goes back to Euler and Gauss and was later generalized using the language of algebraic number theory. We now provide a detailed proof based on the arithmetic of the ring $\mathbb{Z}[\omega]$, where $\omega = \frac{1 + \sqrt{-3}}{2}$.

\subsection*{Norms and Factorization in $\mathbb{Z}[\omega]$}
The ring $\mathbb{Z}[\omega] = \{ a + b\omega : a, b \in \mathbb{Z} \}$ is a Euclidean domain with respect to the norm function:
$N(a + b\omega) = a^2 + ab + b^2.$
This norm is multiplicative, i.e., $N(zw) = N(z)N(w)$, and plays a role analogous to the absolute value in $\mathbb{Z}$. The conjugate of $z = a + b\omega$ is defined as:
$\bar{z} = a + b\bar{\omega} = a - b\omega^2 = a + b - b\omega,$
which also belongs to $\mathbb{Z}[\omega]$.

The units in $\mathbb{Z}[\omega]$ are the six elements:
$\{ \pm 1, \pm \omega, \pm \omega^2 \},$
all of which have norm 1. As a Euclidean domain, $\mathbb{Z}[\omega]$ is also a unique factorization domain (UFD), so every nonzero non-unit element can be written uniquely (up to associates and ordering) as a product of irreducible primes.

\subsection*{Classification of Primes}
The behavior of rational primes in $\mathbb{Z}[\omega]$ falls into three distinct cases:
\begin{itemize}
\item If $p \equiv 1 \mod 3$, then $p$ splits as $p = \rho \bar{\rho}$, where $\rho, \bar{\rho} \in \mathbb{Z}[\omega]$ are distinct non-associate primes, and $N(\rho) = p$.
\item If $p \equiv 2 \mod 3$, then $p$ remains inert in $\mathbb{Z}[\omega]$, i.e., it is a prime in $\mathbb{Z}[\omega]$ and $N(p) = p^2$.
\item The prime $3$ ramifies: $3 = -\omega^2(1 + \omega)^2$. Hence, $1 + \omega$ is a prime of norm 3, and all primes $\rho \in \mathbb{Z}[\omega]$ with norm 3 are associates of $1 + \omega$.
\end{itemize}

Conversely, any prime $\rho \in \mathbb{Z}[\omega]$ falls into one of these cases:
\begin{itemize}
\item $N(\rho) = p$ where $p \equiv 1 \mod 3$, so $\rho$ divides a rational prime that splits,
\item $\rho = p \in \mathbb{Z}$ where $p \equiv 2 \mod 3$, so $\rho$ is a rational prime,
\item $N(\rho) = 3$, in which case $\rho \sim 1 + \omega$ up to unit multiple.
\end{itemize}

\subsection*{Factorization of Integer Norms}
Let $z = a + b\omega \in \mathbb{Z}[\omega]$, and define its norm as
$n = N(z) = z\bar{z} = a^2 + ab + b^2.$
We now describe how the prime factorization of $z \in \mathbb{Z}[\omega]$ determines the prime factorization of $n \in \mathbb{Z}$.

By the multiplicativity of the norm, the norm $n = N(z)$ factors as the product of the norms of the prime elements in the factorization of $z$. This yields the prime factorization of the integer $n$ in $\mathbb{Z}$. Specifically:

\begin{itemize}
\item Any rational prime $q \equiv 2 \mod 3$ that divides $n$ must appear with an even exponent, since it corresponds to an inert prime in $\mathbb{Z}[\omega]$ with norm $q^2$.
\item Any rational prime $p_i \equiv 1 \mod 3$ that appears in the prime factorization of $n$ with exponent $n_i$, must have arisen from splitting: $p_i = \rho_i \bar{\rho}_i$. Then $z$ must contain $\rho_i^{a_i} \bar{\rho}_i^{b_i}$ such that $a_i + b_i = n_i$, and the number of such exponent combinations is $n_i + 1$.
\item The prime 3 ramifies: since $3 =-\omega^2 (1 + \omega)^2$, any power of 3 in $n$ corresponds to an even power of the prime element $1 + \omega$ (up to unit multiple).
\end{itemize}

The six units in $\mathbb{Z}[\omega]$ contribute to counting distinct associate elements $z$, but all these yield the same integer norm $n$. Therefore, each distinct factorization of $z$ corresponds to six distinct ordered pairs $(a,b) \in \mathbb{Z}^2$ such that $N(a + b\omega) = a^2 + ab + b^2 = n$.

Hence, the total number of ordered integer solutions $(a,b)$ to $a^2 + ab + b^2 = n$ is
$6(n_1 + 1)(n_2 + 1) \cdots (n_k + 1).$
This completes the proof of the representation theorem.

\section{Characterization of Nontrivial Solutions}
\label{s1}
We now apply this number-theoretic structure to the problem of identifying when the modified cubic identity \eqref{e1} has nontrivial solutions.

\begin{theorem}\label{th2}
Let $N = n^2 + n + 1$. Then the identity

$$
\sigma^3(n) + x^3 - k^3 = \left( \sigma(n) + x - k \right)^2
$$

has a nontrivial solution $(k,x,n)$ if and only if $N$ has at least two prime factors (counting with multiplicity) congruent to $1 \mod 3$.
\end{theorem}

Recall that expressed in terms of quadratic forms, Theorem~\ref{th2} is equivalent to showing that the quadratic form $a^2+ab+b^2=n^2+n+1$ has positive solutions with $a \ne b+2$ and $(a,b) \ne (n,1)$ if and only if $n^2+n+1$ has at least two prime factors (counting with multiplicity) congruent to $1 \mod 3$. (See Section~\ref{MCI}.)

\begin{lemma}\label{l1}
If $N$ is not a perfect square, then the orbit of any solution of $a^2 + ab + b^2 = N$ contains exactly one positive solution.
\end{lemma}

\begin{proof}
Suppose that $z = a + b\omega$ is a solution of $a^2+ab+b^2 = N$. Since $N$ is not a perfect square, we must have $ab \ne 0$ and $a+b\neq 0$. Consider the orbit

$$
\{ z = a + b\omega,\ -z = -a - b\omega,\ \omega z = -b + (a + b)\omega,\ -\omega z = b - (a + b)\omega,$$$$\ \omega^2 z = -(a + b) + a\omega,\ -\omega^2 z = (a + b) - a\omega \}.
$$

We examine which element of the orbit has both real and imaginary coefficients (in the basis $\{1, \omega\}$) positive.

If $z$ is a positive solution, then it is the only one in its orbit with $a, b > 0$. If $a < 0$ and $b < 0$, then $-z$ is positive. If $a > 0, b < 0$, then among $\omega z$, $\omega^2 z$, one is positive. If $a < 0, b > 0$, then among $-\omega z$, $-\omega^2 z$, one is positive.
\end{proof}

\begin{proof}[Proof of Theorem~\ref{th2}]
For every \( n > 1 \), two distinct representations of \( N = n^2 + n + 1 \) as \( a^2 + ab + b^2 \) are given by \( (a,b) = (n,1) \) and \( (a,b) = (1,n) \).

To determine the number of positive representations of \( N  \) as \( a^2 + ab + b^2 \), we begin by applying Theorem~\ref{th1}. This theorem gives the total number of such representations as
\[
6(n_1 + 1)(n_2 + 1) \cdots (n_k + 1),
\]
where the \( p_i \equiv 1 \mod 3 \) are the prime divisors of \( N \), and the \( n_i \) are their respective exponents. The factor of 6 arises from the fact that \( \mathbb{Z}[\omega] \) has six units; that is, all solutions \( (a + b\omega) \) to the norm equation are grouped into orbits under multiplication by units, and each such orbit consists of exactly six distinct representations of \( N \).

Since \( N = n^2 + n + 1 \) lies strictly between \( n^2 \) and \( (n+1)^2 \), it is never a perfect square. This allows us to apply Lemma~\ref{l1}, which states that each orbit contains exactly one solution with \( a > 0 \) and \( b > 0 \). It follows that the number of positive representations of \( N \) is
\[
m = (n_1 + 1)(n_2 + 1) \cdots (n_k + 1).
\]

When \( m = 2 \), the only positive representations are \( (n,1) \) and \( (1,n) \), both of which are trivial. If \( m = 3 \), then there is one additional positive solution \( (a, b) \). If this solution satisfies \( a = b + 2 \), then it is again trivial — but then by symmetry the pair \( (b, a) \) would also be a solution, giving four total representations, contradicting \( m = 3 \). So the third solution must be nontrivial. In all other cases, where \( m \geq 4 \), at least one nontrivial solution exists.
\end{proof}

\section{When \( n^2 + n + 1 \) Is Prime or Three Times a Prime}

We are now ready to prove the result stated in the introduction and originally announced without a proof in \cite{fejzic2005}. Rather than proving it directly, we obtain it as a corollary of Theorem~\ref{th2}, after establishing a key number-theoretic lemma that is of independent interest.

\begin{lemma}\label{lem:prime_divisors}
If a prime \( p \) divides \( n^2 + n + 1 \), then either \( p = 3 \), or \( p \equiv 1 \mod 3 \).
Moreover, if \( p = 3 \), then it appears with exponent exactly one in the factorization of \( n^2 + n + 1 \).
\end{lemma}

\begin{proof}
Suppose \( p \mid n^2 + n + 1 \), so that
\[
n^2 + n + 1 \equiv 0 \pmod{p}.
\]
Since \( p \nmid n \), multiplying both sides by \( n \) yields
\[
n(n^2 + n + 1) \equiv 0 \mod p \quad \Rightarrow \quad n^3 \equiv 1 \mod p.
\]
Hence, the order of element \( n \in \mathbb{Z}_p^\times \) divides 3. If \( n \equiv 1 \mod p \), then substituting gives
\[
1^2 + 1 + 1 = 3 \equiv 0 \mod p \Rightarrow p = 3.
\]
If \( n \not\equiv 1 \mod p \), then \( n \) has order exactly 3 in the group \( \mathbb{Z}_p^\times \), which has order \( p - 1 \). By Lagrange’s Theorem, the order of any element divides the group order, so
\[
3 \mid (p - 1) \quad \Rightarrow \quad p \equiv 1 \mod 3.
\]

To show that 3 cannot appear with exponent greater than one, assume \( n \equiv 1 \mod 3 \), so \( n = 3k + 1 \). Then:
\[
n^2 + n + 1 = (3k + 1)^2 + (3k + 1) + 1 = 9k^2 + 6k + 1 + 3k + 1 + 1 = 3(3k^2 + 3k + 1),
\]
where the second factor is not divisible by 3. Therefore, \( 3 \) appears to exponent exactly one in the prime factorization.
\end{proof}

\begin{corollary}\label{c}
Let \( N = n^2 + n + 1 \). Then the identity
\begin{equation}\label{11}
\sigma^3(n) + x^3 - k^3 = \left( \sigma(n) + x - k \right)^2
\end{equation}
has a nontrivial solution \( (k,x,n) \) if and only if \( N \) is not equal to a prime or three times a prime.
\end{corollary}

\begin{proof}
By Lemma~\ref{lem:prime_divisors}, every prime divisor of \( N \) is either congruent to \( 1 \mod 3 \) or equal to 3, and 3 appears with exponent at most one. Thus, \( N \) has exactly one such prime factor (counted with multiplicity) if and only if \( N=3 \), or \(N=p\) where  a prime \( p \equiv 1 \mod 3 \),  or \( N = 3p \) for such a prime \( p \). In all other cases, \( N \) has at least two prime factors \( \equiv 1 \mod 3 \), counted with multiplicity. By Theorem~\ref{th2}, this is equivalent to the existence of a nontrivial solution to the identity \eqref{11}.
\end{proof}


\bigskip

\medskip

\end{document}